\documentclass[12pt]{ourlematema} 
\usepackage{amsmath, amsthm, amssymb, hyperref, color}
\usepackage{MnSymbol} % for \dashedrightarrow command
\usepackage{graphicx}
\usepackage{caption}
\usepackage[all]{xypic}
\usepackage{verbatim}
\usepackage{chemfig, chemnum}
\usepackage{tikz}
\usetikzlibrary {arrows}

\usepackage[linesnumbered,lined,commentsnumbered]{algorithm2e}
\usepackage{caption}
\usepackage{subcaption}
\usepackage{makecell}
\usepackage{array}

\newcommand{\RR}{\mathbb{R}}
\newcommand{\bbS}{\mathbb{S}}
\newcommand{\CC}{\mathbb{C}}
\newcommand{\PP}{\mathbb{P}}
\newcommand{\ZZ}{\mathbb{Z}}

\newcommand{\Lcal}{\mathcal{L}}
\newcommand{\adjHP}{\mathcal{A}}
\newcommand{\tr}{\mathrm{tr}}
\newcommand{\rk}{\mathrm{rk}}
\newcommand{\GL}{\mathrm{GL}}
\newcommand{\K}{K}

\newcommand{\adj}{\mathrm{adj}}

\newcommand{\CH}{\mathrm{CH}}
\newcommand{\join}{\mathrm{join}}
\makeatletter
\newcommand{\xdashrightarrow}[2][]{\ext@arrow 0359\rightarrowfill@@{#1}{#2}}
\def\rightarrowfill@@{\arrowfill@@\relax\relbar\rightarrow}
\def\arrowfill@@#1#2#3#4{%
  $\m@th\thickmuskip0mu\medmuskip\thickmuskip\thinmuskip\thickmuskip
   \relax#4#1
   \xleaders\hbox{$#4#2$}\hfill
   #3$%
}
\makeatother

\definecolor{forest}{RGB}{11,128,35}

 \title{The maximum likelihood degree \\ of linear spaces of symmetric matrices}
 
 \titlemark{The ML degree of Linear Spaces of Symmetric Matrices}
 
  \author{Carlos Am\'endola}
  \address{%
  Technical University of Munich, Ulm University \\
\email{carlos.amendola@tum.de}
}

\author{Lukas Gustafsson}
\address{%
KTH Royal Institute of Technology \\
\email{lukasgu@kth.se}
}

\author{Kathl\'en Kohn}
\address{%
KTH Royal Institute of Technology\\
\email{kathlen@kth.se}
}

\authorline

\author{Orlando Marigliano}
\address{%
KTH Royal Institute of Technology\\
\email{orlandom@kth.se}
}

\author{Anna Seigal}
\address{%
University of Oxford \\
\email{seigal@maths.ox.ac.uk}
}

\authormark{Am\'endola \ - \ Gustafsson \ - \  Kohn \ - \ Marigliano \ - \ Seigal}

\date{2020/11/11}

\begin{document}

\maketitle

\begin{abstract} 
We study multivariate Gaussian models that are described by linear conditions on the concentration matrix. We compute the maximum likelihood (ML) degrees of these models. 
That is, we count the critical points of the likelihood function over a linear space of symmetric matrices. 
We obtain new formulae for the ML degree, one via line geometry, and another using Segre classes from intersection theory. We settle the case of codimension one models, and characterize the degenerate case when the ML degree is zero.
\end{abstract}

\section{Introduction}

We study $n$-dimensional multivariate Gaussian distributions with mean zero. Every such distribution is described by the covariance matrix or, its inverse, the concentration matrix. 
Both matrices lie in the
cone of $n \times n$ positive definite matrices.
We consider multivariate Gaussian models as a set of concentration matrices in the cone, and focus on linear models that are homogeneous (i.e. if some concentration matrix is in the model, then so are its scalar multiples). The log-likelihood of observing sample covariance $S$ at a concentration matrix $K$ is
\begin{align} 
\label{eqn:ell}
\ell_S (K) = \log \det (\K) - \tr (\K S).
\end{align} 

The maximum likelihood (ML) degree counts the complex critical points of the log-likelihood function, as we vary over the Zariski closure of the model. %\cite{catanese2006maximum}. 
 The Zariski closure of a linear model is a linear space of symmetric matrices $\Lcal$
in the space $\bbS^n$ of $n \times n$ complex symmetric matrices.
Linear spaces are well-known to have unique minimizers of the Euclidean distance function, but the same is not true of the log-likelihood. %We will see examples of linear spaces in $\bbS^n$ whose ML degree ranges from $0$ to~$n-1$. 

The maximum likelihood degree of a Gaussian linear concentration model was first studied in \cite{sturmfels2010multivariate}. For generic linear spaces $\Lcal$, the ML degree equals the degree of the reciprocal variety $\Lcal^{-1}$ (defined below). This degree is non-trivial to compute in general, and recently a connection to the space of complete quadrics has led to more tools, including a proof of its polynomiality in the ambient dimension $n$~\cite{michalek2020maximum,manivel2020}. 

However, Gaussian statistical models used in practice are seldom generic. For example, a natural family of linear concentration models are undirected Gaussian graphical models. In this setting, $\Lcal$ is defined by zeros at the entries that correspond to missing edges from a graph. Several results and conjectures for the ML degrees of special classes of graphs can be found in \cite{sturmfels2010multivariate, uhlergeoMLE}. The question of finding the ML degree of any pencil, that is, of a linear space of dimension 2, has recently been answered in \cite{pencils}, using Segre symbols\footnote{named after Corrado Segre (1863-1924).}.

In this paper we consider arbitrary linear spaces of symmetric matrices of any dimension. Our main results include several characterizations of their ML degree, in particular a formula based on line geometry (Theorem \ref{thm:MLproduct}) and a formula based on intersection theory (Theorem \ref{segre-classes-formula}). The latter is given in terms of Segre classes\footnote{named after Beniamino Segre (1903-1977).}. Section \ref{sec:hyperplanes} is devoted to the hyperplane case. When the hyperplane is defined by an annihilator matrix $A$, we prove that its ML degree equals the rank of $A$ minus one (Proposition \ref{prop:hyperplane}). A complete classification of the ML degrees for all linear spaces in $\bbS^3$ is provided in Section \ref{sec:space}. We study linear spaces with ML degree 0 in Section \ref{sec:mldeg0} and give several equivalences for this degenerate case to occur (Theorem \ref{thm:mlDegreeZero}). 

\section{Likelihood geometry} 

In this section we lay out our geometric set-up. 
We then give two approaches to compute the ML degree of a linear space, and illustrate them on an example. 

\subsection{Geometric set-up}
\label{sec:geometry}

Let $\Lcal \subseteq \bbS^n$ be 
 a linear space of symmetric matrices.
 Throughout this article, we assume that $\Lcal$ is \emph{regular}, i.e. that it contains at least one matrix of full rank.
 We consider the inner product on the real points in $\bbS^n$
 given by $\langle K, \Sigma \rangle = \tr (K \Sigma)$. We extend the trace pairing to all complex matrices in $\bbS^n$. 
  This allows us to define the \emph{annihilator} or \emph{polar space} of $\Lcal$:
  $$ \Lcal^\perp := \{ \Sigma \in \bbS^n \mid \tr (  K \Sigma ) = 0 \text{ for all } K \in \Lcal \} . $$
The derivative of the log-likelihood $\ell_S: \Lcal \to \mathbb{R}$ in \eqref{eqn:ell} at a real point $K \in \Lcal$ is 
\begin{align*} D_\K (\ell_S) : \Lcal & \longrightarrow \RR \\
\dot{\K} & \longmapsto \langle \K^{-1} - S , \dot{\K} \rangle  .\end{align*} 
Hence, the condition for $\K$ to be a critical point is given by 
$ \K^{-1} - S \in \Lcal^\perp. $
We define the \emph{reciprocal variety} $\Lcal^{-1}$ to be the Zariski closure in $\bbS^n$ of the inverses of the invertible matrices in $\Lcal$.
Then the \emph{maximum likelihood degree}  of $\Lcal \subseteq \mathbb{S}^n$ is the number of matrices, for generic $S \in \mathbb{S}^n$, in the intersection
\begin{align}
\label{eq:MLdeg}
    \mathcal{L}^{-1} \cap \left( \mathcal{L}^\perp + S \right).
\end{align}
Given a matrix $\Sigma$ in the intersection, the corresponding critical point in $\Lcal$ is $\Sigma^{-1}$.

\begin{rem}
The reader might worry about some matrices in the intersection being singular, since only the invertible matrices in $\Lcal^{-1}$ correspond to some $\K \in \Lcal$. One might also ask whether points should be counted with multiplicity. Later in this section (in Lemmas~\ref{lemma:1} and~\ref{lemma:2}, and Proposition~\ref{prop:projection}) we see:
\begin{enumerate}
\item For generic $S$, all intersections are at invertible matrices.
\item 
For generic $S$, the points in the intersection in~\eqref{eq:MLdeg} occur without multiplicity. That is, the definition of ML degree does not depend on if we count intersection points with or without multiplicity.
\end{enumerate} 
\end{rem}

We now move the set-up to the projective space $\PP \bbS^n$. 
The varieties $\Lcal^{-1}$ and  $\Lcal^\perp$ are both defined by homogeneous polynomials. We denote their projectivizations by $L^{-1}$, $L^\perp$.
We let $\Lcal^\perp_S$ denote the span of $\mathcal{L}^\perp$ and $S$, and denote its projectivization by $L^\perp_S$.
The dimension of $L^\perp_S$ equals the codimension of $L^{-1}$. Hence $L^\perp_S$ and $L^{-1}$ meet either at $\deg(L^{-1})$ many points, counted with multiplicity, or at infinitely many points. 
We consider the projection from $L^\perp$:
\begin{align}
\begin{split}
\label{eqn:pi_map}
    \pi_{L^\perp} : \PP \bbS^n &\,\xdashrightarrow{\hspace{0.5cm} } \left\lbrace
    W \in \mathrm{Gr}(\dim L^\perp + 1,  \PP\bbS^n) \mid \, L^\perp \subset W
    \right\rbrace \cong \mathbb{P}^{\dim L}, \\
    S &\;\;\longmapsto L^{\perp}_S,
    \end{split}
\end{align}
where $\mathrm{Gr}(k,\mathbb{P}^N)$ denotes the Grassmannian of $k$-dimensional subspaces of $\mathbb{P}^N$. We will show in Proposition~\ref{prop:projection} that the ML degree is the degree of $\pi_{L^\perp} |_{L^{-1}}$, the projection map $\pi_{L^\perp}$ restricted to the reciprocal variety $L^{-1}$.

\begin{lemma}
\label{lemma:1}
Let $X \subset \PP^N$ be an irreducible variety of dimension $d$, and let 
$\pi_V: \PP^N \dashrightarrow \PP^d$ be the projection  from a linear space $V \in \mathrm{Gr}(N-d-1, \PP^N)$.
Then the generic fiber of $\pi_V |_X$ is reduced.
\end{lemma}

\begin{proof}
If the restricted map $\pi_V |_X$  is not dominant, i.e. its generic fiber is empty,  the assertion is trivial.
So we assume the generic fiber is finite and non-empty.

Since the map $\pi_V$ restricted to the singular locus of $X$ is not dominant, the generic fiber of $\pi_V |_X$ does not contain singular points of $X$.
Thus, the generic fiber of $\pi_V |_X$ is \emph{not} reduced if and only if 
the generic $W \in \mathrm{Gr}(N-d, \PP^N)$ containing $V$ 
intersects $X$ non-transversely at a smooth point $x \in X$ outside of $V$
(i.e., $x \in W \cap \mathrm{Reg}(X) \setminus V$ and $W + \mathbb{T}_xX \neq \PP^N$, where $\mathbb{T}_xX \subset \PP^N$ denotes the embedded tangent space of $X$ at $x$).

Since we assumed the map $\pi_V |_X$ to be dominant, 
the join of $X$ and $V$ is the whole ambient space $\PP^N$.
By Terracini's lemma~\cite[Corollary 1.11]{aadlandsvik}, 
we have that $V + \mathbb{T}_x X = \PP^N$ for generic $x \in X$.
So $Y := \overline{\left\lbrace x \in \mathrm{Reg}(X) \setminus V \mid V + \mathbb{T}_xX \neq \PP^n \right\rbrace} \subset X$
is a proper subvariety.
Since we assumed $X$ was irreducible, we have $\dim Y < \dim X$ and $\pi_V |_Y$ is not dominant.
This means that the generic $W \in \mathrm{Gr}(N-d, \PP^N)$ containing $V$ does not pass through any point in $Y \setminus V$,
so it cannot intersect $X$ non-transversely outside of $V$.
\end{proof}

Lemma~\ref{lemma:1} implies that the fibers of the map  $\pi_{L^\perp} |_{L^{-1}}$ are generically reduced.
That is, points in the generic fiber are present without multiplicity. 

\begin{lemma}
\label{lemma:2}
The generic fiber of $\pi_{L^\perp} |_{L^{-1}}$ consists only of invertible matrices.
\end{lemma} 

\begin{proof}
We may assume that  $\pi_{L^\perp} |_{L^{-1}}$ is dominant since otherwise the assertion is trivial.
Hence, the generic fiber is the intersection
\begin{align}
\label{eq:genFiber}
( L^\perp_S \cap L^{-1} ) \setminus L^\perp  \text{ for generic } S \in \PP \bbS^n.  
\end{align} The singular matrices in the intersection are the fiber of $L^{\perp}_S$ under the restriction of $\pi_{L^\perp}$ to the locus $L^{-1} \cap Z(\det)$. The latter is a proper subvariety of $L^{-1}$ and thus $\pi_{L^\perp} |_{L^{-1} \cap Z(\det)}$ is not dominant. 
Hence the generic fiber of $\pi_{L^\perp} |_{L^{-1}}$ does not contain singular matrices.
\end{proof}

\begin{prop}
\label{prop:projection}
The ML degree of $\mathcal{L}$
is the degree of the restricted map $\pi_{L^\perp} |_{L^{-1}}$,
i.e. the generic number of intersection points of $L^\perp_S$
and $L^{-1}$ that do not lie on~$L^\perp$. 
\end{prop}

\begin{proof}
Both domain and codomain of the map $\pi_{L^\perp} |_{L^{-1}}$ have the same dimension as $L$, so the map is generically finite. 
The degree of the map is the cardinality of the generic fiber \eqref{eq:genFiber}.
The fiber of $L^{\perp}_S$ under $\pi_{L^\perp} |_{L^{-1}}$ can be lifted to affine space by setting the coefficient of $S$ to be one. 
This affine lift is the intersection in~\eqref{eq:MLdeg}, up to multipicity and removing singular matrices.
Hence the cardinality of the fiber is equal to the ML degree, because the generic fiber is reduced, by Lemma~\ref{lemma:1}, and only contains invertible points, by Lemma~\ref{lemma:2}.
\end{proof}

We note that the degree of the projection $\pi_{L^\perp} |_{L^{-1}}$ is sometimes used as the definition of the ML degree, see~\cite[Definition 1.1]{manivel2020}, \cite[Definition 2.3]{michalek2020maximum}, and~\cite[Definition 5.4]{michalek2016exponential}.  Proposition~\ref{prop:projection} and Lemma~\ref{lemma:2} combine to show that for generic $S \in \bbS^n$, all intersection points of $\Lcal^\perp + S$ and $\Lcal^{-1}$ occur at invertible matrices.
Throughout the paper, we will make use of the following Lemma, proved in~\cite[Lemma 4.1]{pencils}.

\begin{lemma}
\label{lem:congruence}
The ML degree of a linear subspace $\Lcal \subset \bbS^n$ only depends on its congruence class under change of basis by $\GL_n$.
\end{lemma}

\subsection{A first formula}

We now give a first approach to compute the ML degree of a linear space.
We recall from Proposition~\ref{prop:projection} that the
ML degree of $\Lcal$ is the number of intersection points of $L^\perp_S$
and $L^{-1}$ that do not lie on $L^\perp$. 
The following result shows that all the intersection points in $L^\perp$ are non-invertible matrices. Then Lemma~\ref{lemma:2} implies that the matrices we seek to exclude from the intersection $L^{\perp}_S \cap L^{-1}$ are exactly the non-invertible matrices.

\begin{lemma}
\label{lem:singular}
Every point in $\mathcal{L}^{-1} \cap \mathcal{L}^\perp$ is a non-invertible matrix.
\end{lemma}

\begin{proof}
Assume that an invertible matrix $\Sigma$ is contained in $\mathcal{L}^{-1} \cap \mathcal{L}^\perp$.
This implies $\Sigma^{-1} \in \mathcal{L}$, and since $\Sigma \in \mathcal{L}^\perp$ we derive the contradiction 
$0 = \tr(\Sigma^{-1}\Sigma) = n$.
\end{proof}

\begin{prop}
\label{prop:bigint} 
Let $\{A_1, \ldots, A_c \}$ be a basis for $\Lcal^\perp \subset \bbS^n$, and define
$$ \adjHP_\Lcal =  \{ \Sigma \in \bbS^n \mid \tr(A_i \cdot \adj(\Sigma)) =0 \text{ for }  i = 1, \ldots, c \} .$$
Then the ML degree of $\Lcal$ is the number of invertible matrices in the intersection
$ \adjHP_\Lcal \cap ( \Lcal^\perp + S) $ for generic $S \in \bbS^n$. 
\end{prop}

\begin{proof}
We have  $\Lcal^{-1} \subseteq \adjHP_\Lcal$, since $\tr ( A_i \cdot \adj(\K^{-1}))=0$ for $\K \in \Lcal$, but the inclusion may be strict.
However, all points in the difference $\adjHP_\Lcal \setminus \Lcal^{-1}$ are non-invertible matrices. Indeed, if we have some invertible $\Sigma \in \adjHP_\Lcal$, then the defining equations of $\adjHP_\Lcal$ imply that $\tr ( A_i \cdot \Sigma^{-1})$ vanishes, i.e. $\Sigma^{-1} \in \Lcal$. 
For generic $S$, none of the critical points of the likelihood occur at singular matrices, by Lemma~\ref{lemma:2}, hence this includes all critical points of the likelihood. 
\end{proof}

The inclusion $\Lcal^{-1} \subseteq \adjHP_\Lcal$ will always be strict if the dimension of $\Lcal$ is small enough. This is because all defining equations of $\adjHP_\Lcal$ vanish when $\rk(\Sigma) \leq n-2$. 
Hence if the corank-two matrices are not in $\Lcal^{-1}$, then these lie in the difference $\adjHP_\Lcal \setminus \Lcal^{-1}$.
The advantage of the larger intersection in Proposition~\ref{prop:bigint} over the smaller intersection in~\eqref{eq:MLdeg} is that we have defining equations for both sides. 
This enables us to obtain the following.

\begin{prop}
\label{prop:firstapproach}
Let $\{A_1, \ldots, A_c \}$ be a basis for $\Lcal^\perp \subset \bbS^n$.
Then the ML degree of $\Lcal$ is the number of invertible matrices $\sum_i t_i A_i + S$ that are critical points of $\ell(t_1, \ldots, t_c) := \det(\sum_i t_i A_i + S) $.
 \end{prop}
 
 \begin{proof}  We have the identity
 $ \tr \left( A_i \cdot \adj \left( \sum_j t_j A_j + S\right) \right) = \frac{d}{dt_i} \left( \det(\sum_j t_j A_j + S) \right) .$
 So the invertible critical points count the invertible matrices in the intersection of $\adjHP_\Lcal$ and $\Lcal^\perp + S$. Then we conclude using Proposition~\ref{prop:bigint}.
 \end{proof} 
 
We note that this proposition has a natural connection to the problem of maximizing the determinant along a spectrahedron, which computes the MLE in the real setting, see \cite{sturmfels2010multivariate}. 
 
 \begin{exa}
 \label{ex:running_example}
Let
 $
 \Lcal = \{ \K \in \bbS^3 \mid \kappa_{11} = \kappa_{22} = 0 \}$.
 A basis of $\Lcal^\perp$ is given by
  $A_1 = \left[\begin{smallmatrix} 1&0&0 \\ 0&0&0 \\ 0&0&0 \end{smallmatrix}\right]$ and   
 $A_2 = \left[\begin{smallmatrix} 0&0&0 \\ 0&1&0 \\ 0&0&0 \end{smallmatrix}\right] . 
 $
 The ML degree of $\Lcal$ is the number of invertible matrices $S + t_1 A_1 + t_2 A_2$ that are critical points of $ \ell(t_1, t_2) = \det(S + t_1 A_1 + t_2 A_2) =
s_{33}t_1t_2 +t_1(s_{22}s_{33} - s_{23}^2) + t_2(s_{11}s_{33} - s_{13}^2) + \det(S)
$. We obtain the conditions
$$
s_{33}t_2 + (s_{22}s_{33} - s_{23}^2) = 0 \qquad 
s_{33}t_1 + (s_{11}s_{33} - s_{13}^2) = 0.
$$
This system has a unique solution for $(t_1, t_2)$ for generic $S$. The last step is to verify that the critical point is at an invertible matrix. 
We substitute  our expressions for $t_1$ and $t_2$ into $s_{33}\ell(t_1,t_2)$ and obtain for generic $S$ the expression
$
s_{33}\det(S) -(s_{22}s_{33} - s_{23}^2)(s_{11}s_{33} - s_{13}^2) \neq 0.
$
Hence the model has ML degree 1. 
\end{exa}

\subsection{A line geometry formula}
\label{sec:2.3}
In this subsection we give a formula for the ML degree of $\Lcal$ based on the Grassmannian $\mathrm{Gr}(1, \PP \mathbb{S}^n)$.

\begin{lemma}
\label{lem:lines_and_param} 
Let $\Lcal \subset \bbS^n$ be a linear subspace, and fix $S \in \PP\bbS^n$ generic.
The ML degree of $\Lcal$ is the number of pairs $(\Sigma,\Gamma) \in L^{-1} \times L^\perp$ such that $\Sigma \neq \Gamma$ and $S$ is on the line $\ell \in \mathrm{Gr}(1, \PP \mathbb{S}^n)$ spanned by $\Sigma$ and $\Gamma$. 
\end{lemma} 

\begin{proof}
The ML degree of $\Lcal$ is the number of matrices in the intersection of $L^{-1} \cap L^\perp_S$ that are not in $L^\perp$, see Proposition~\ref{prop:projection}.
A point $\Sigma$ in the intersection $L^{-1} \cap L_S^\perp$, but not in $L^\perp$, can be written as a linear combination of some $\Gamma \in L^\perp$ and $S$, where the coefficient of $S$ is non-zero. That is, the point $\Sigma$ is on a line $\ell \in \mathrm{Gr}(1, \PP \mathbb{S}^n)$ spanned by $\Gamma \neq \Sigma$ and $S$. 
Hence the ML degree counts the $\Sigma \in L^{-1}$ that lie on a line spanned by $S$ and some $\Gamma \in L^\perp$ distinct from $\Sigma$. 
This is equivalent to the assertion.
\end{proof}

\begin{figure}[h]
    \centering
    \includegraphics[width=6cm]{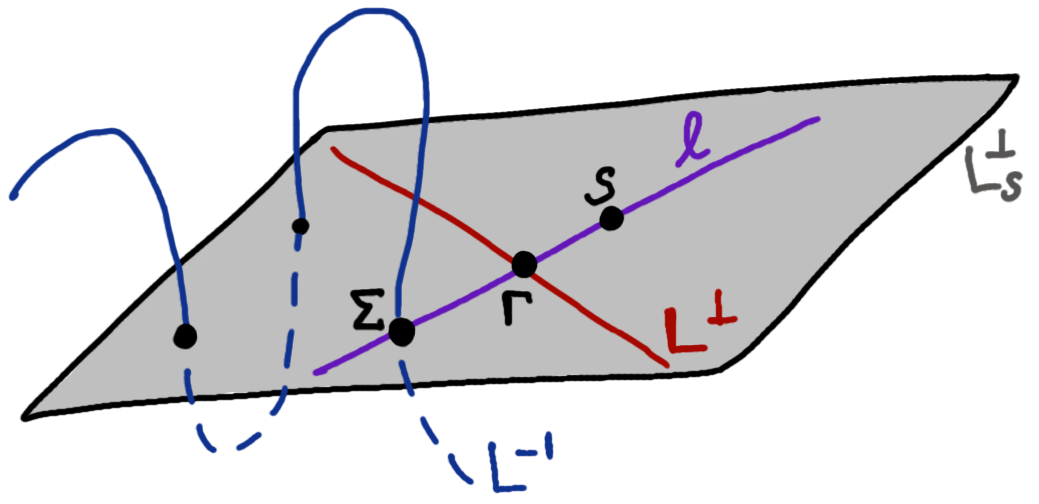}
    \caption{A diagram to show the construction of lines $\ell$ spanned by $(\Sigma,\Gamma) \in L^{-1} \times L^\perp$, containing some $S \in \PP\bbS^n$. The ML degree of $\Lcal$ is the total number of intersections of $L^{-1}$ with the linear space $L_S^\perp$.}
    \label{fig:linesMN}
\end{figure}

We consider the 
Schubert variety of lines passing through $S$:
$$ \mathcal{G}_S := \left\lbrace \ell \in \mathrm{Gr}(1, \PP \mathbb{S}^n) \mid S \in \ell \right\rbrace . $$
We are interested in lines $\ell \in \mathcal{G}_S$ that are spanned by $\Sigma \in L^{-1}$ and $\Gamma \in L^\perp$, by Lemma~\ref{lem:lines_and_param}. For this, we introduce the variety $\mathcal{J}$ in $\mathrm{Gr}(1, \PP \mathbb{S}^n)$:
\begin{align*}
    \mathcal{J} := \overline{\left\lbrace
\ell \in \mathrm{Gr}(1, \PP \mathbb{S}^n) \mid \exists (\Sigma , \Gamma) \in L^{-1} \times L^\perp: \Sigma \neq \Gamma, \Sigma \in \ell, \Gamma \in \ell
\right\rbrace}.
\end{align*}
Note that the union of the lines $\ell$ in $\mathcal{J}$ is the join of $L^{-1}$ and $L^\perp$ in $\PP \bbS^n$.

\begin{thm}
\label{thm:MLproduct}
Let $\Lcal \subset  \bbS^n$ be a linear subspace of codimension at least two, and
let $S \in \PP \mathbb{S}^n$ be generic.
Then the ML degree of $\mathcal{L}$ is
$|\mathcal{J} \cap \mathcal{G}_S|$. 
\end{thm}

We prove Theorem~\ref{thm:MLproduct} by first counting the number of parametrizations of a general line $\ell \in \mathcal{J}$. 
This number is the degree of the following projection. 

\begin{align}
\begin{split}
\label{eqn:gamma} 
    \gamma: 
    \overline{{\left\lbrace
    (\Sigma,\Gamma,\ell) \in L^{-1} \times L^\perp \times \mathrm{Gr}(1, \PP \mathbb{S}^n)
    \mid \Sigma \neq \Gamma, \; \Sigma \in \ell, \; \Gamma \in \ell
    \right\rbrace}} &\longrightarrow \mathcal{J} 
    ,\\
    (\Sigma,\Gamma,\ell) & \longmapsto \ell.
    \end{split} 
\end{align}

\begin{lemma}
\label{lem:degGammaIsOne}
Let $\Lcal \subset \bbS^n$ have codimension at least two and non-zero ML degree. Then $\gamma$ is a birational map. 
\end{lemma}

\begin{proof}
Since the ML degree is non-zero, a general $S \in \PP\bbS^n$ lies on a line spanned by some $\Sigma \in L^{-1}$ and $\Gamma \in L^\perp$. That is, the join of $L^{-1}$ and $L^\perp$ fills the ambient space $\PP \bbS^n$, see also Theorem~\ref{thm:mlDegreeZero}(iii). In particular, $L^{-1}$ is not a cone over $L^\perp$ (otherwise their join would be $L^{-1}$), so the map $\gamma$ is generically finite. Since $L^\perp$ is a linear space, the degree of $\gamma$ counts the intersections of a generic line $\ell \in \mathcal{J}$ with $L^{-1}$.

We consider a generic line $\ell \in \mathcal{J}$.
The span of the line with $L^\perp$ gives the linear space $L^\perp_S$ where $S$ is a generic point.
Since the codimension of the model $\Lcal$ is at least two, the linear space $L^\perp_S$ strictly contains $\ell$.
The reciprocal variety $L^{-1}$ intersects $L^\perp_S \setminus L^\perp$ at finitely many points, where their number is the ML degree of $\Lcal$, by Proposition~\ref{prop:projection}.
Since the line $\ell$ is generic, it  passes through exactly one of the points.
Otherwise, if $\ell$ is spanned by $\Sigma \in L^{-1}$ and $\Gamma \in L^\perp$, we can perturb the point $\Gamma$ on $L^\perp$ (which has positive dimension by our assumption on the codimension of $\Lcal$) to obtain a new line that only meets $L^{-1}$ at $\Sigma$.
\end{proof}

\begin{proof}[Proof of Theorem~\ref{thm:MLproduct}.]
We saw in Lemma~\ref{lem:lines_and_param} that the ML degree is, for gene\-ric~$S$, the number of pairs $(\Sigma,\Gamma) \in L^{-1} \times L^\perp$ with $\Sigma \neq \Gamma$ and $S$ on the line spanned by $\Sigma$ and $\Gamma$.
Hence, the ML degree is zero if and only if $\mathcal{J} \cap \mathcal{G}_S = \emptyset$. 

It remains to consider those $\Lcal$ with non-zero ML degree. 
In particular, $L^{-1}$ is not a cone over $L^\perp$. This implies that the variety $\mathcal{J}$ has both dimension and codimension \mbox{$\dim \PP \mathbb{S}^n-1$} inside the Grassmannian $\mathrm{Gr}(1, \PP \mathbb{S}^n)$.
The Schubert variety $\mathcal{G}_S$ has the same (co)dimension. 
Hence the intersection $\mathcal{J} \cap \mathcal{G}_S$ is finite for generic $S$, for instance as a consequence of~\cite[Theorem~1.7]{eisenbud20163264}.
A line $\ell$ in this intersection is spanned by a unique pair $(\Sigma, \Gamma) \in L^{-1} \times L^\perp$ due to Lemma~\ref{lem:degGammaIsOne} and the genericity of $S$. Hence, the assertion follows from Lemma~\ref{lem:lines_and_param}.
\end{proof}

In the next section we discuss hyperplanes, where we see that the degree of the projection~$\gamma$ can exceed one, and in fact  $\deg(\gamma)$ equals the ML degree. We conclude this section by revisiting Example~\ref{ex:running_example}; we use Theorem~\ref{thm:MLproduct} to compute its ML degree. 

\begin{exa}
\label{ex:running_exampleAgain}
We consider $\Lcal = \{ K \in \bbS^3 \mid \kappa_{11} = \kappa_{22} = 0 \}$.
The ML degree is the number of lines  $ \ell \in \Gamma \cap \mathcal{G}_S$ for generic $S$, by Theorem~\ref{thm:MLproduct}. We fix a generic $S$, with entries $s_{ij}$, and show that there is only one line $\ell$, spanned by $\Sigma \in L^{-1}$ and $\Gamma \in L^\perp$, with $S \in \ell$. We first express $\Sigma$ in terms of $S$. 
Since matrices in $L^\perp$ are supported only on the $(1,1)$ and $(2,2)$ entries, we have $\sigma_{ij} = s_{ij}$ for all entries except possibly $\sigma_{11}$ and $\sigma_{22}$. But $\Sigma^{-1} \in L$ means $\sigma_{ii} \sigma_{33} - \sigma_{i3}^2 = 0$ for $i=1,2$. Since $s_{33} \neq 0$ by genericity of $S$, we recover all entries of $\Sigma$ uniquely from $S$. The line spanned by $\Sigma$ and generic $S$ meets the linear space $L^\perp$ at the unique point $\Gamma$. 
\end{exa}

\section{Hyperplanes}
\label{sec:hyperplanes} 

In this section, we find the ML degree of hyperplanes in $\bbS^n$ via two methods: by finding defining equations of $\Lcal^{-1}$, and by the line geometry formula from Section~\ref{sec:2.3}. We also compute the ML degree of hyperplanes in the space of diagonal matrices via these two methods. We confirm that our results agree with formulae for the ML degree of a diagonal linear model via matroids, see~\cite[Section 3]{sturmfels2010multivariate}.

We saw in Section~\ref{sec:2.3} that hyperplanes are excluded from the statement of Theorem~\ref{thm:MLproduct}: for hyperplanes, the projection map $\gamma$ need not be birational. In fact, here we see that $\deg(\gamma)$ is equal to the ML degree. 

\begin{prop}
\label{prop:deggamma}
Let $\Lcal \subset \bbS^n$ be a hyperplane with non-zero ML degree. Then the ML degree of $\Lcal$ is the degree of the projection $\gamma$ in~\eqref{eqn:gamma}. 
\end{prop}

\begin{proof}
The ML degree is the number of intersection points of a generic line $\ell$ passing through the point $L^\perp$
with $L^{-1} \setminus L^\perp$, by Proposition~\ref{prop:projection}.
Since this number is non-zero by assumption, 
the line $\ell$ is a generic point on $\mathcal{J}$. 
Hence, the degree of $\gamma$ is also the cardinality of the intersection of $\ell$ with $L^{-1} \setminus L^\perp$.
\end{proof}

\begin{rem}
When $\Lcal \subset \bbS^n$ is a hyperplane with non-zero ML degree, $L^\perp$ is a point and $\mathcal{J} \cap \mathcal{G}_S$ consists of the unique line spanned by $L^\perp$ and the generic point~$S$.
Hence, for \emph{any} regular linear subspace $\Lcal \subset \bbS^n$, 
Theorem~\ref{thm:MLproduct} and Proposition~\ref{prop:deggamma} combine to show that the ML degree of $\Lcal$ is 
\begin{equation}
    \label{eqn:mlproduct}
    | \mathcal{J} \cap \mathcal{G}_S| \cdot \deg(\gamma),
\end{equation}
for generic $S$, 
using the convention $0 \cdot \infty = 0$.
Indeed, the ML-degree is zero if and only if $\mathcal{J} \cap \mathcal{G}_S =\emptyset$.
In the case of non-zero ML degree, 
$|\mathcal{J} \cap \mathcal{G}_S|  = 1$
if $\Lcal$ is a hyperplane, otherwise $\deg(\gamma) = 1$ by Lemma~\ref{lem:degGammaIsOne}.
\end{rem}

We now compute the ML degree of a hyperplane. 
We write the linear equation defining the hyperplane as $\tr(A \K) = 0$, where $A$ is a fixed complex symmetric matrix.
We first obtain a description of the reciprocal variety $\Lcal^{-1}$ for a hyperplane $\Lcal$. 

\begin{lemma}
\label{lem:inverseeq}
Consider the hyperplane $\Lcal = \{ \K : \tr(A\K) = 0 \}$. The variety $\mathcal{L}^{-1}$ is a hypersurface defined by the irreducible degree $n-1$ polynomial $ \tr(A \cdot  \operatorname{adj}(\K))$.
\begin{proof}
The polynomial $f(\K):= \tr(A \cdot  \operatorname{adj}(\K))$ defines $\adjHP_\Lcal = Z(f)$, which contains $\Lcal^{-1}$. 
The variety $\Lcal^{-1}$ is a hypersurface, since matrix inversion is a birational map.
Hence $\Lcal^{-1}$ is an irreducible component of $\adjHP_\Lcal$.

To conclude, we show that the polynomial $f$ is irreducible. Matrices in $\adjHP_\Lcal \setminus \Lcal^{-1}$ must be singular and, since $\adjHP_\Lcal$ is a hypersurface, the existence of some $\Sigma \in \adjHP_\Lcal \setminus \Lcal^{-1}$ implies the existence of a codimension one locus of singular matrices, i.e. the locus $Z(\det)$. 
However, the determinant does not divide $f(\K)$ since the determinant has degree $n$, while $f$ has degree $n-1$. It remains to show that $f$ is not a power of a lower degree polynomial. For this, we observe that the diagonal entries $\kappa_{ii}$ of $\K$ are only present with linear exponent in $f$. (The off diagonal entries may be present with higher power because we are in the space of symmetric matrices.) For example, we can write $f = \kappa_{ii} b_i(\K) + c_i(\K)$, where $b_i$ and $c_i$ do not involve the variable $\kappa_{ii}$. This cannot be a power of a polynomial $g(\K)$ unless $b_i(\K) = 0$ for all $i$. 
But we have the equality $b_i(K) = \tr(A_i \cdot \adj(K_i))$, where $A_i$ and $K_i$ denote the submatrices of $A$ and $K$ without row and column $i$.
Hence, the polynomial $b_i$ is zero if and only if $A_i$ is zero.
Repeating this for three values of $i$ shows that all entries of $A$ must vanish, a contradiction. It remains to consider the case $n=2$. Here, $f$ is linear, hence irreducible. 
\end{proof}
\end{lemma}

We now prove our main result of the section.

\begin{prop}
\label{prop:hyperplane} 
Let $\Lcal$ be a hyperplane in $\bbS^n$ defined by the equation $\tr ( A \K)=~0$, for some non-zero matrix $A \in \bbS^n$. Then the ML degree of $\Lcal$ is $\rk(A) -1$.
\end{prop}

\begin{proof}
We count points in the intersection of the hypersurface $\Lcal^{-1}$ with $\Lcal^\perp + S$. The linear space $\Lcal^\perp + S$ is the line $tA + S$. 
Since the polynomial in Lemma~\ref{lem:inverseeq} is the defining equation of $\Lcal^{-1}$,
the intersection is the values of $t$ such that 
$$ \tr( A \cdot \adj(S + tA))  = 0 .$$
 The number of values of $t$ for which this condition holds is the ML degree of $\Lcal$. Moreover, by Lemma~\ref{lemma:1} the points are present without multiplicity for generic $S$. 
We observe that
$$ \tr( A \cdot \adj(S + tA))  = 
    \frac{d}{dt} \left[ \det(S + tA)\right] .$$
It therefore suffices to show that the degree of the polynomial $\det(S + tA)$ is $r:= \rk(A)$. 

The ML degree is unchanged under congruence, by Lemma~\ref{lem:congruence}. We know from~\cite[Equation (1.1)]{bunse1988singular} that every complex symmetric matrix is congruent to a diagonal matrix with entries $(1,\ldots,1,0,\ldots,0)$ on the diagonal. Hence we can assume the matrix $A$ defining the hyperplane $\Lcal$ has this form. This shows that the polynomial has rank at most $r$ in $t$. The coefficient of $t^r$ is the determinant of the submatrix of $S$ on rows $r+1, \ldots, n$ and columns $r+1, \ldots, n$, if $r < n$, or the coefficient is $1$ if $r = n$. Since $S$ is generic, the polynomial $\det(S + tA)$ has degree $r$, as required.
\end{proof}

\begin{rem}
We give an alternative proof of Proposition~\ref{prop:hyperplane} using Proposition~\ref{prop:deggamma}, for $\Lcal$ with non-zero ML degree.
We count the number of intersection points of $L^{-1}$ with a generic $\ell \in \mathrm{Gr}(1, \PP \bbS^n)$ passing through $A$.
Since the degree of $L^{-1}$ is $n-1$, by Lemma~\ref{lem:inverseeq}, there are $n-1$ intersection points in total, counted with multiplicity. The intersection multiplicity at the point $A$ is $n-r$, as follows. 
The homogeneous polynomial $\tr(A \cdot \adj(sS + tA))$ is divisible by $s^{n-r}$, by the congruence argument in the proof of Proposition~\ref{prop:hyperplane}.
The ML degree is the number of intersection points away from $A$, which is is equal to $(n-1) - (n-r) = r-1$.
\end{rem}

\begin{exa}
If $\Lcal$ is defined by $\tr(A \K)$ where $A$ has rank one, the linear space $\Lcal$ has ML degree zero. We study ML degree zero examples in Section~\ref{sec:mldeg0}.
\end{exa}

\noindent
We consider models defined by hyperplanes in the space of diagonal matrices.

\begin{prop}
\label{prop:diagonal_hyperplane} 
Consider a regular linear subspace $\Lcal \subset \bbS^n$ defined by $\kappa_{ij} = 0$ for all $i \neq j$ together with the condition $\tr ( A \K ) = 0$ for a non-zero diagonal matrix $A \in \bbS^n$. Then the ML-degree of $\Lcal$ is $\rk(A) -1$. 
\end{prop}

\begin{proof}
We first show that the degree of $L^{-1}$ is $\rk(A)-1$.
We can change basis under congruence action so that $A$ is the matrix with $r:= \rk(A)$ ones on its diagonal, and all other entries zero. We let the diagonal coordinates of an $n \times n$ matrix be given by the variables $x_1, \ldots, x_n$. Then the variety $L$ is defined by 
$$ x_1 + \cdots + x_r = 0 .$$
The variety $L^{-1}$ is also contained in the space of diagonal matrices. It is defined by the condition
$$ \frac{1}{x_1} + \cdots + \frac{1}{x_r} = 0 . $$
We multiply by the product $x_1 \cdots x_r$ to obtain a hypersurface $V$ defined by an irreducible polynomial of degree $r-1$. We now exclude the possibility that $V$ strictly contains $L^{-1}$. A matrix in $V \setminus L^{-1}$ is non-invertible. Hence if $L^{-1} \subsetneq V$ then $V$ must contain a hypersurface of non-invertible diagonal matrices, i.e. $V$ must contain a coordinate hyperplane. We see from its defining equation that $V$ does not contain a coordinate hyperplane, hence $V = L^{-1}$. 

We show that the ML degree of $\Lcal$ agrees with the degree of $L^{-1}$. 
It suffices to show that $L^{-1} \cap L^\perp$ is empty,
by Proposition~\ref{prop:projection}. The variety $L^{-1}$ is contained in the diagonal matrices, and the only diagonal matrix in $L^\perp$ is $A$. We conclude by observing that $A \notin L^{-1}$, by setting $x_1 = \ldots = x_r = 1$ into the equation for $L^{-1}$.
\end{proof} 

\begin{rem}
\label{rem:3.8}
We give an alternative proof of Proposition~\ref{prop:diagonal_hyperplane}, based on Theorem~\ref{thm:MLproduct}.
That is, we count the lines $\ell$ spanned by $(\Sigma,\Gamma) \in L^{-1} \times L^\perp$ passing through a generic $S$. We seek the $\Gamma \in \Lcal^\perp$ such that $S-\Gamma \in \Lcal^{-1}$. The off diagonal entries of $\Gamma$ must match those of $S$, since $\Lcal^{-1}$ is contained in the diagonal matrices, so it suffices to look at diagonal entries. The diagonal entries of $\Gamma$ are those of $tA$ for some scalar $t$. As before, we work up to congruence, and assume that $A$ is diagonal with $r:= \rk(A)$ ones on the diagonal. Then the condition $S - \Gamma \in \Lcal^{-1}$ gives the following degree $r-1$ polynomial in $t$:
$$ \sum_{i=1}^r \left( \prod_{j \neq i} (s_j - t) \right) ,$$
where $s_j$ is the $j$th diagonal entry of $S$. Hence a generic $S$ lies on $r-1$ lines.
\end{rem} 

We note that the same ML degree of $\rk(A)-1$ appears in both Propositions~\ref{prop:hyperplane} and~\ref{prop:diagonal_hyperplane}. 
However, the two occurrences of $\rk(A)-1$ come from different parts of the multiplicative formula for the ML degree in~\eqref{eqn:mlproduct}. 
For a hyperplane, there is a unique line for each $S$ but $\rk(A)-1$ parametrizations of each line. In comparison, for a hyperplane in the diagonal matrices, each $S$ lies on $\rk(A)-1$ lines, each with a unique parametrization.

\bigskip

We now describe how Proposition~\ref{prop:diagonal_hyperplane} follows from more general results: a Gr\"{o}bner basis for $\Lcal^{-1}$ from~\cite{proudfoot2004broken}, and a formula for the ML degree of $\Lcal$ from the characteristic polynomial of its associated matroid, see~\cite[Theorem 2.1(a)]{eur2020reciprocal}. 
We identify the space of diagonal $n \times n$ matrices with $\CC^n$, and view a linear space of diagonal matrices as $\Lcal \subseteq \CC^n$. 
The inverse variety $\Lcal^{-1}$ is then also contained in the diagonal matrices.  The inverse $\Lcal^{-1}$ does not intersect the polar space $\Lcal^\perp$, see~\cite[Corollary 3.3]{sturmfels2010multivariate}. Hence the ML degree of $\Lcal$ is $\deg(\Lcal^{-1})$, by Proposition~\ref{prop:projection}.

ML degrees were connected to matroids in~\cite[Section 3]{sturmfels2010multivariate}. 
A brief introduction to matroids is given in~\cite{ardila2018geometry}.
A matroid $M$ is pair $(E,\mathcal{I})$, where $E$ is a finite set, and $\mathcal{I}$ a collection of subsets of $E$, called its independent sets, which satisfy certain axioms. A matroid can also be defined by its circuits, other subsets $C \subseteq E$ that satisfy certain other axioms.
We briefly describe how to associate a matroid $M$ to a linear space $\Lcal \subseteq \CC^n$. We take $E = [n]$. The matroid $M$ has circuits given by minimal subsets $C \subseteq [n]$ such that a linear combination of $\{ x_c : c \in C \}$ vanishes on $\Lcal$.  
That is, the circuits of $\Lcal$ are the supports of minimal support vectors in $\Lcal^\perp$, see~\cite[Theorem 3.2]{sturmfels2010multivariate}.
For example, if $\Lcal$ consists of all vectors orthogonal to $e_1 + \cdots + e_r$, then $M$ has just one circuit, $\{ 1, \ldots, r \}$. 

Assume the linear combination of $\{ x_c : c \in C \}$ that vanishes on $\Lcal$ is $\sum_{c \in C} a_c x_c$. Following~\cite{proudfoot2004broken}, we define the polynomial
$$ f_C := \sum_{c \in C} a_c \left( \prod_{c' \in C \backslash \{ c \} } x_{c'} \right) .$$
The polynomials $f_C$, as $C$ ranges over circuits of $M$, gives a universal Gr\"{o}bner basis for the ideal defining $\Lcal^{-1}$, see~\cite[Theorem 4]{proudfoot2004broken}. In the special case where the linear combination is $x_1 + \cdots + x_r$, the polynomial $f_C$ is constructed in the proof of Proposition~\ref{prop:diagonal_hyperplane}.

If $\Lcal$ is a hyperplane in the diagonal matrices then, up to congruence, it is defined by the vanishing of $x_1 + \cdots + x_r$, where $r$ is the rank of the diagonal matrix in $\Lcal^\perp$. The result~\cite[Theorem 4]{proudfoot2004broken} implies that $\Lcal^{-1}$ is described by the vanishing of $\sum_{i=1}^r \left( \prod_{1 \leq j \leq r, j \neq i} x_j \right)$. This polynomial has degree $r-1$, hence $\deg(\Lcal^{-1})$ is $r-1$, and the ML degree of $\Lcal$ is also $r-1$. 
Together with Remark~\ref{rem:3.8}, this gives a third proof of Proposition~\ref{prop:diagonal_hyperplane}.
We conclude this section with a fourth proof, obtained by specializing a formula for the ML degree in terms of the characteristic polynomial of its associated matroid, see~\cite[Theorem 2.1(a)]{eur2020reciprocal}.

The characteristic polynomial of the matroid $M$ is:
$$ \chi_M ( \lambda) := \sum_{S \subseteq [n]} {(-1)}^{|S|} \lambda^{r(M) - r(S)} , $$
where $r(M)$ is the rank of $M$, the size of its maximal independent sets, and $r(S)$ is the rank of the submatroid on $S$. A subset  is independent in the submatroid on $S$ if it is independent in $M$ and contained in $S$. The constant term $\chi_M(0)$ is then the number of subsets of $[n]$ whose restriction has the same rank as $M$.
We have $\deg(\Lcal^{-1}) = |\chi_M(0)|$, see~\cite[Theorem 2.1(a)]{eur2020reciprocal}.
The invariant $\chi_M(0)$ is sometimes called the M\"obius invariant of the matroid $M$, and denoted $\mu(M)$; it is mistakenly referred to as the beta invariant in~\cite{sturmfels2010multivariate}.

We evaluate $\chi_M(0)$ when $\Lcal$ is a hyperplane in the diagonal matrices. 
As above, we work up to congruence and assume $\Lcal$ has normal vector $e_1 + \cdots + e_r$. The rank of $M$ is $n-1$, hence a subset $S \subset [n]$ can only be a submatroid of the same rank if $|S| \geq n-1$. There is one choice with $|S| = n$. It remains to count the $S$ of size $n-1$. There can be no circuits in the submatroid, so we must have removed one of the first $r$ coordinates. Hence there are $r$ choices. We obtain $\chi_M(0) = (-1)^{m-1} (r-1)$. Hence $| \chi_M(0) | = r-1$.

\section{An intersection theory formula}

In this section we give a formula for computing the ML degree that does not involve calculations with generic matrices $S$, unlike the ones so far.

\emph{Intersection theory} is used throughout algebraic geometry to obtain answers to many kinds of counting problems. Of central importance is the \emph{Chow ring} of a smooth variety, a graded ring whose elements can be thought of as generalized subvarieties organized by their dimension. The graded parts of the Chow ring are called the \emph{Chow groups}. For instance, the Chow ring of $\mathbb P^N$ is the polynomial ring $\ZZ[\zeta]/\zeta^{N+1}$, where an element of the form $k\zeta^j$ represents a generic codimension-$j$ subvariety of degree $k$. Multiplication in the Chow ring corresponds to taking scheme-theoretic intersections of subvarieties.

Let $X_1$ and $X_2$ be irreducible subvarieties in $\mathbb P^N$ of complementary dimension. 
We consider the diagonal $\Delta \cong \mathbb P^N$ of $\PP^N \times \PP^N$ and let $X_1 \cap X_2 = \Delta \cap (X_1 \times X_2)$.
Let $\beta$ be the dimension of $X_1\cap X_2$ and $\CH_j(X_1\cap X_2)$ its $j$-th Chow group for $j=0,\dotsc,\beta$. 
The $j$-th \emph{Segre class} of $X_1\cap X_2$ in $X_1\times X_2$
is denoted
\[
s^j(X_1\cap X_2, X_1\times X_2)
\in \CH_j(X_1\cap X_2),
\]
and defined in \cite[Ch.\ 7, \S 4.2]{fulton}. We let
$\sigma^j(X_1\cap X_2, X_1\times X_2)$ denote the degree of the $j$th Segre class, taken by the inclusion of $X_1\cap X_2$ in the diagonal $\Delta \cong \mathbb P^N$. 
The function \texttt{segre(Z,V)} in the \texttt{Macaulay2} \cite{M2} package \texttt{SegreClasses}
can be used to compute the Segre classes of a subscheme $Z$ of a scheme $V$ that lives in a product of projective spaces~\cite{harris2020segre}.
The following lemma describes how to multiply classes of varieties in terms of Segre classes.

\begin{lemma}\label{segre-lemma}
Let $X_1$ and $X_2$ be irreducible varieties in $\mathbb P^N$ of complementary dimension. Let $\beta$ be the dimension of the intersection $X_1\cap X_2$. We have the following equality in the $0$-th Chow group of $X_1 \cap X_2$:
\begin{equation}\label{desired-formula}
	X_1\cdot X_2 = \sum_{j=0}^\beta \binom{N+1}{j} s^j(X_1\cap X_2, X_1\times X_2) \cdot \zeta^j.
\end{equation}
\end{lemma}
\begin{proof}
Both sides of~\eqref{desired-formula} are additive over connected components, as follows. 
We have $X_1\cdot X_2 = \sum_C  (X_1\cdot X_2)^C$, where the sum runs over the connected components of $X_1\cap X_2$. 
The additivity on the right hand side follows from the fact that Segre classes are additive over connected components: if $Z \subset V$ decomposes as $Z_1 \dot\cup Z_2$ then $s^j(Z,V) = s^j(Z_1, V) + s^j(Z_2, V)$. 
Hence we may assume that $Z = X_1 \cap X_2$ is connected.

Following \cite[9.1]{fulton}, we take $Y = \mathbb P^N \times \mathbb P^N$, let $X = \Delta$ be the diagonal of $Y$, and let $V = X_1\times X_2$.
The normal bundle $\mathcal N$ of $\Delta$ in $Y$ is the tangent bundle of $\mathbb P^N$, so $c(\mathcal N) = (1 + \zeta)^{N+1}$, where $c$ denotes the total Chern class~\cite[Example 3.2.11]{fulton}.
Now, \cite[Prop.~9.1.1]{fulton} gives 
\begin{align}
\label{eq:zeroTermSegreClasses}
    (X_1\cdot X_2)^Z = (\Delta \cdot V)^Z
= \{(1+\zeta)^{N+1}\cdot s(Z, V)\}_0,
\end{align}
where $s(Z, V) = \sum_{j=0}^\beta s^j(Z,V)$ is the total Segre class and $\{ \cdots \}_0$ denotes the terms that belong to $\CH_0$. Collecting these terms, we obtain the formula~\eqref{desired-formula}.
\end{proof}

\begin{thm}\label{segre-classes-formula}
Let $N = \dim \PP \bbS^n$. The ML degree of a regular  subspace $\mathcal L \subset \bbS^n$~is
\begin{equation}\label{new-formula}
\deg L^{-1} - \sum_{j=0}^{\beta}\binom{N}{j}\sigma^j(L^{-1}\cap L^\bot, L^{-1}\times L^\perp).
\end{equation}

\end{thm}
\begin{proof}
We apply Lemma~\ref{segre-lemma} with $X_1 = L^{-1}$ and $X_2 = L_S^\perp$ embedded in $ \mathbb{PS}^n \cong \mathbb P^N $.
By Proposition~\ref{prop:projection}, the class $L^{-1}\cdot L_S^\perp$ decomposes as the sum of a class supported in $L^{-1}\cap L^{\bot}$ and a class supported in the finite set $E$ of critical points of the log-likelihood.
Since $\deg (L^{-1}\cdot L_S^\perp) = \deg L^{-1}$, taking degrees in \eqref{desired-formula} gives
\[
\deg L^{-1} = \sum_{j=0}^\beta \binom{N+1}{j} \sigma^j(L^{-1}\cap L^\perp, L^{-1}\times L_S^\perp)  + \sum_{P\in E}\sigma^0(P, L^{-1}\times L_S^\perp).
\]
The latter term is the ML degree of $\mathcal L$ because $\sigma^0(P,L^{-1}\times L_S^\perp)$ is the multiplicity of $L^{-1}\times L_S^\perp$ along $P$ \cite[4.3]{fulton}, which is one for all $P\in E$ since $S$ is generic (by Lemma~\ref{lemma:1} and the fact that each $P$ is smooth on $L^{-1}$).
The former term is the degree zero part of the class $(1+\zeta)^{N+1}\cdot s(L^{-1}\cap L^\perp, L^{-1} \times L_S^\perp)$ as in~\eqref{eq:zeroTermSegreClasses}.

The projection $L_S^\perp\setminus \{S\} \to L^\perp $
away from $S$ identifies $L_S^\perp\setminus \{S\}$ as the hyperplane bundle $\mathcal{O}_{L^\perp}(1)$.
By~\cite[Example 4.2.7]{fulton} we have
\begin{align*}
s(L^{-1}\cap L^\perp, L^{-1}\times L^\perp)
&=
c(\mathcal O(1)) \cdot 
s(L^{-1}\cap L^\perp, L^{-1}\times (L^\perp_S\setminus\{S\}))\\
&=
(1 + \zeta) \cdot 
s(L^{-1}\cap L^\perp, L^{-1}\times L^\perp_S),
\end{align*}
where the second equality follows from $c(\mathcal O (1)) = 1 + \zeta$ and \cite[Proposition~4.2]{fulton}.
Thus we have the following equality of terms in $\CH_0$
\[
\{
(1+\zeta)^{N+1}\cdot s(L^{-1}\cap L^\perp, L^{-1} \times L_S^\perp)
\}_0
=
\{
(1+\zeta)^N\cdot s(L^{-1}\cap L^\perp, L^{-1} \times L^\perp)
\}_0.
\]
Expanding the right hand side, we obtain the term that is subtracted in~\eqref{new-formula}. 
\end{proof}

Formula~\eqref{new-formula} simplifies when the intersection $L^{-1} \cap L^\perp$ is finite and only contains smooth points of the reciprocal variety $L^{-1}$.
The following immediate corollary is used in~\cite{dye2020maximum} to compute the ML degrees of all three-dimensional subspaces of $\bbS^3$ (also listed in Section~\ref{sec:space}.3).

\begin{cor}
\label{cor:simpleSegre}
Let $\Lcal \subset \bbS^n$ be a linear space such that the intersection $L^{-1} \cap L^\perp$ is finite and consists only of smooth points of $L^{-1}$.
Then the ML degree of $\Lcal$ is
\begin{align*}
    \deg L^{-1} - \deg (L^{-1} \cap L^\perp),
\end{align*}
where the second term is the scheme-theoretic degree of the intersection $L^{-1} \cap L^\perp$
(i.e., the constant coefficient of its Hilbert polynomial).
\end{cor}

\begin{exa}
Let $\Lcal \subset \bbS^3$ be a four-dimensional subspace whose polar space $\Lcal^\perp$ is a regular pencil spanned by a rank-one and a rank-two matrix. 
This pencil has Segre symbol $[(1 1) 1]$ (see Section~\ref{sec:space}.4).
Up to congruence, the pencil is spanned by $\left[\begin{smallmatrix}
1&0&0 \\ 0&1&0 \\ 0&0&0
\end{smallmatrix} \right]$ and $\left[\begin{smallmatrix}
0&0&0 \\ 0&0&0 \\ 0&0&1
\end{smallmatrix} \right]$.
We compute the ML degree of $\Lcal$ using our intersection theory formula in Theorem~\ref{segre-classes-formula}. A \texttt{Macaulay2} computation reveals that the reciprocal variety $L^{-1}$ has degree $4$.
Next, we apply the function \texttt{segre(Z,V)} 
for $V = L^{-1} \times L^\perp$
and $Z = \Delta \cap V$
to obtain 
\begin{align*}
    \mathrm{segre}(Z,V) = 2H_1^5H_2^5,
\end{align*}
where $H_1$ and $H_2$ are the hyperplane classes in the Chow rings of the factors of $\PP^5 \times \PP^5$.
This corresponds to the Segre class $2 \zeta^5$ in the Chow ring of the diagonal $\Delta \cong \PP^5$.
Hence, $\sigma^0(Z,V) = 2$ and $\sigma^j (Z,V) = 0$ for all $j > 0$, so Theorem~\ref{segre-classes-formula} tells us that the ML degree of $\Lcal$ is $\deg L^{-1} - \sigma^0(Z,V) = 4-2 = 2$. We include this computation in our supplementary code~\cite{ourcode}. 

However, we cannot apply the simplified version of the formula in Corollary~\ref{cor:simpleSegre}, even though $L^{-1} \cap L^\perp$ is a single point.
This is because the point is singular on the reciprocal variety $L^{-1}$.
In fact, $L^{-1}$ is singular along two lines and two isolated points. 
The two lines intersect exactly at the point $L^{-1} \cap L^\perp$.
The scheme-theoretic degree of the intersection $L^{-1} \cap L^\perp$ is in fact $1$, which shows that the formula in Corollary~\ref{cor:simpleSegre} does not hold in this case.
\end{exa}

\begin{exa}
We revisit the four-dimensional linear space $\mathcal L\subset \mathbb S^{3}$ in Examples \ref{ex:running_example} and \ref{ex:running_exampleAgain}.
Its polar space $L^\perp$ is a singular pencil with Segre symbol  $[1 1 ; ; 1]$ (see Section~\ref{sec:space}.4).
Up to congruence, $\Lcal$ is the only linear subspace of $\bbS^3$ with non-zero ML degree such that the intersection $L^{-1} \cap L^\perp$ is not finite.

In fact, the reciprocal variety $L^{-1}$ is singular along a plane that contains the line $L^\perp$.
The singular plane contains two other embedded lines that meet $L^\perp$ in two points. 
Using our Segre classes approach, we can determine how much this singular structure contributes to the degree of $L^{-1}$.

Let $V = L^{-1} \times L^\perp$ and $Z = \Delta \cap V$. 
The function \texttt{segre(Z,V)} yields the Segre class $-7\zeta^5 + 2\zeta^4$ in the Chow ring of $\Delta \cong \PP^5$.
Hence, $\sigma^0(Z,V) = -7$,  $\sigma^1(Z,V) = 2$, and $\sigma^j(Z,V) = 0$ for all $j>1$.
Applying Theorem~\ref{segre-classes-formula}, we see that the ML degree of $\Lcal$ is 
$\deg L^{-1} - \sigma^0(Z,V) - 5 \sigma^1(Z,V) = 4 + 7 - 10 = 1$.
The details of this computation can be found in our supplementary code~\cite{ourcode}. 
\end{exa}

\begin{rem}
In fact, we first encountered the formula~\eqref{new-formula} in \cite{aadlandsvik} where it is used to determine if the join of two projective varieties has the expected dimension.
Our multiplicative formula~\eqref{eqn:mlproduct} for the ML degree is exactly the degree of the class appearing in \cite[Proposition~2.2(i)]{aadlandsvik} (using the substitution $r \mapsto 2$, $X_1 \mapsto L^{-1}$, $X_2 \mapsto L^\perp$, $n \mapsto N-1$, $m \mapsto n+r-2 = N-1$).
In the proof of~\cite[Theorem~2.4]{aadlandsvik} it is shown that that degree is given by our formula~\eqref{new-formula}.
This gives an alternative argument for Theorem~\ref{segre-classes-formula}. 
\end{rem}

We now turn our attention to a statistically meaningful example. Formula \ref{segre-classes-formula} allows us to give an explanation for the ML degree of the smallest Gaussian graphical model with ML degree greater than $1$, namely, the model associated to the undirected 4-cycle in Figure~\ref{fig:4cycle}.

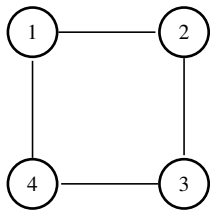
\begin{figure}[h]
    \centering
    \begin{tikzpicture}[-,>=stealth',shorten >=1pt,auto,node distance=2cm,semithick]
    \tikzstyle{every node}=[circle,line width =1pt,font=\scriptsize,minimum height =0.65cm]
        \node (i1) [draw] {1};
        \node (i2) [right of = i1, draw] {2};
        \path (i1) edge (i2);
        
        \node (i3) [below of = i2, draw] {3};
        \path (i2) edge (i3);
        
        \node (i4) [left of = i3, draw] {4};
        \path (i3) edge (i4);
        \path (i4) edge (i1);
    \end{tikzpicture}
    \caption{The undirected 4-cycle.}
    \label{fig:4cycle}
\end{figure}

\begin{exa}\label{ex:4cycle}
The linear space corresponding to the Gaussian 4-cycle model is
$ \Lcal = \{ \K \in \bbS^4 \mid \kappa_{13} = \kappa_{24} = 0 \}$, since the edges $1-3$ and $2-4$ are missing from the graph (see Figure~\ref{fig:4cycle}).
The polar space $L^\perp$ is a regular pencil that intersects the reciprocal variety $L^{-1}$ at two points. 
Both points are singular on $L^{-1}$ and a \texttt{Macaulay2} computation using the function \texttt{segre(Z,V)}, where $V = L^{-1} \times L^\perp$ and $Z = \Delta \cap V$, reveals that each point contributes $2$ to the 0-th Segre class: $\sigma^0(Z,V) = 2+2=4$.
Formula \ref{segre-classes-formula} then computes the ML degree to be equal to $\deg L^{-1} - \sigma^0(Z,V) = 9-(2+2) = 5$.
\end{exa}

We believe that understanding the intersection theory behind larger $n$-cycles could shed light on a 2008 conjecture concerning their ML degree: it is conjectured in \cite[Section 7.4]{drton2008lectures} that the ML degree of $\Lcal_n$ is $(n-3)2^{n-2}+1$,
where $\Lcal_n$ is the linear space associated to the Gaussian $n$-cycle model. 

The intersection $L_4^{-1} \cap L_4^\perp$ from the $4$-cycle example above is a \emph{monomial scheme}, i.e. its defining ideal is generated by monomials.
When investigating the $5$-cycle, we see that the same is true for the intersection $L_5^{-1} \cap L_5^\perp$. We conjecture that the intersection $L_n^{-1} \cap L_n^\perp$ is a monomial scheme for all $n$.

The computations in Example \ref{ex:4cycle} quickly become prohibitive for larger $n$ with the function \texttt{segre}. However, there is an alternative geometric way of interpreting and computing Segre classes of monomial schemes due to Aluffi \cite{aluffi2016segre}. 
This approach expresses Segre classes of \emph{regular crossings monomial schemes} as integrals over polytopal Newton regions. 
These integrals can be efficiently computed by triangulating the Newton regions.
Unfortunately, this method does not apply as is to our 4-cycle example since the singularities described in Example~\ref{ex:4cycle} interfere with the regular crossings assumption. 
We expect a generalization of the technique in~\cite{aluffi2016segre} to be a promising way to make progress towards the conjecture in~\cite{drton2008lectures}.

\section{Full classification for $n=3$} \label{sec:space}

We compute the ML degree of every regular subspace in $\mathbb{S}^3$. Here are the results listed by dimension.

\paragraph{\textbf{1) Lines.}}
A linear space spanned by a full-rank matrix has ML degree one.

\paragraph{\textbf{2) Planes.}}
There are five congruence classes of 2-dimensional regular linear spaces in $ \mathbb{S}^3$. The ML degrees are listed in~\cite[Example 1.3]{pencils} by Segre symbol.
\begin{center}

\begin{tabular}{c|ccccc}
    & $[1~1~1]$ & $[2~1]$ & $[(1~1)~1]$ & $[3]$  & $[(2~1)]$ \\ \hline
  $\deg {L}^{-1}$   & 2 & 2 & 1 & 2 & 1 
  \\
  mld($\Lcal$)             & 2 & 1 & 1 & 0 & 0 
\end{tabular}
\end{center}

\paragraph{\textbf{3) 3-Planes.}}
There are 13 types of 3-dimensional regular linear spaces in $\mathbb{S}^3$ described in~\cite{wall}.
Their ML degrees are computed in \cite[Table 1]{dye2020maximum}:
\begin{center}
    \begin{tabular}{c|ccccccccccccc}
     & $A$ & $B$ & $B^\ast$ & $C$ & $D$ & $D^\ast$ & $E$ & $E^\ast$ & $F$ & $F^\ast$ & $G$ & $G^\ast$ & $H$\\ \hline
  $\deg {L}^{-1}$   & 4 & 3 & 4 & 3 & 2 & 4 & 1 & 4 & 2 & 2 & 1 & 2 & 1 
  \\
  mld($\Lcal$)                & 4 & 3 & 3 & 2 & 2 & 2 & 1 & 1 & 0 & 1 & 0 & 0 & 0
\end{tabular}
\end{center}

\paragraph{\textbf{4) 4-Planes.}}
The congruence classes of 4-dimensional linear spaces are in one-to-one correspondence with the congruence classes of 2-dimensional linear spaces by polarity (via the trace pairing). 
Hence there are 8 such congruence classes according to their polar Segre symbol. Using the representatives in \cite[Table 0]{wall}, we compute their ML degrees in \texttt{Macaulay2}:
\begin{center}
\begin{tabular}{c|cccccccc}
    & $[1 1 1]$ & $[2 1]$ & $[(1 1) 1]$ & $[3]$ & $[(2 1)]$ & $[; 1 ;]$ & $[1 1 ; ; 1]$ & $[2 ; ; 1]$\\ \hline
  $\deg {L}^{-1}$   & 4 & 4 & 4 & 4 & 4 & 1 & 4 & 1 
  \\
  mld($\Lcal$)             & 4 & 3 & 2 & 2 & 1 & 1 & 1 & 0
\end{tabular}
\end{center}
These computations are included in our supplementary code~\cite{ourcode}. 

\paragraph{\textbf{5) Hyperplanes.}}
The ML degree of a hyperplane $\{K \in \mathbb{S}^3 \mid \tr(A K) = 0 \} $ is $\mathrm{rk}(A)-~1$, by Proposition~\ref{prop:hyperplane}.

\section{Maximum likelihood degree zero}
\label{sec:mldeg0}

Let $\mathcal L$ be a regular linear space of symmetric matrices.
As seen in Section~\ref{sec:geometry},
the ML degree is the number of invertible matrices in the intersection $  L^{-1}\cap L^\perp_S$ for a generic matrix $S \in \mathbb{S}^n$.
The case of ML degree is zero is very special.
It implies that none of the matrices in $\mathcal L$ are positive definite, see Corollary \ref{prop:nointerior}. 
Hence $\mathcal L$ does not define a statistical model.
Geometrically, $\mathcal L$ belongs to a special type of degenerate linear spaces, and must satisfy the following equivalent conditions.

\begin{thm}
\label{thm:mlDegreeZero}
Let $\mathcal L\subset \mathbb S^n$ be a regular subspace. The following are equivalent:

(i)\ \ \ The ML degree of $\Lcal$ is zero.

(ii)\ \ The restriction $\pi_{L^\perp} |_{ L^{-1}}$ of the projection in~\eqref{eqn:pi_map} is not dominant.

(iii) The join of $L^{-1}$ and $L^\perp$ is not the whole ambient space $\PP \mathbb{S}^n$.

(iv)\ \,A generic $\K \in \Lcal$ satisfies $(\K \Lcal^\perp \K)  \cap \Lcal \neq \lbrace 0 \rbrace$.

(v)\ \,For every pair of bases  $\{A_i\}_{i=1}^c$ of $\Lcal^\perp$ and $\{B_k\}_{k=1}^d$ of  $\Lcal$,
we have the vanishing of the polynomial $\det(M) \in \mathbb{C}[s_1,\dotsc, s_d]$, where the matrix $M$ has 
$$
M_{ij} = \sum_{k, l=1}^d  s_ks_l \cdot \tr(A_iB_kA_jB_l)
\in \mathbb C[s_1,\dotsc, s_d].
$$
\end{thm}

\begin{proof}
The first two conditions are equivalent by Proposition~\ref{prop:projection}.
For conditions (ii) and (iii), we abbreviate $\pi_{\mathcal L^{\bot}}$ to $\pi$. We have
\[
\pi(\join( L^{-1},  L^{\bot}))
=
\pi( L^{-1})
\quad \text{and}\quad
\join(L^{-1}, L^\perp)
=
\pi^{-1}
\overline{\pi (L^{-1})}.
\]
The first relation shows that (ii) implies (iii). The second relation shows the converse. For the equivalence of (iii) and (iv), observe that $\join(L^{-1},L^\perp)\neq\PP \bbS^n$ if and only if $\overline{\mathcal L^{-1} + \mathcal L^\perp}\neq \bbS^{n}$.
A generic point $\Sigma\in \mathcal L^{-1} + \mathcal L^\perp$ is of the form 
$\Sigma = \K^{-1} + \Sigma'$ where $\K \in \Lcal$ is an invertible matrix 
and $\Sigma' \in  \Lcal^\perp$.
By Terracini's lemma \cite[Corollary~1.11]{aadlandsvik}, 
$
\mathbb{T}_\Sigma(\mathcal{L}^{-1} + \mathcal L^\perp)
= \mathbb{T}_{\K^{-1}}{\Lcal^{-1}}
+ \mathbb{T}_{\Sigma'}{\Lcal^\perp}$.
By matrix calculus, we see that  $\mathbb{T}_{\K^{-1}}{\Lcal^{-1}} = \K^{-1} \Lcal \K^{-1}$.
Hence, (iii) is equivalent to $\mathbb{T}_\Sigma(\mathcal{L}^{-1} + \mathcal L^\perp) = \K^{-1} \Lcal \K^{-1} + \Lcal^\perp$ being contained in a hyperplane,
which means that 
%$\Psi^{-1} \Lcal \Psi^{-1}$ and $\Lcal^\perp$ span at most a hyperplane, i.e.
the intersection $\K^{-1} \Lcal \K^{-1} \cap \Lcal^\perp$ is non-zero. This is equivalent to condition~(iv). 

For the equivalence of (iv) and (v), write $\K = \sum_k s_k B_k$ for a generic $\K\in \mathcal L$ and $\Sigma = \sum_j t_j A_j$ for $\Sigma \in \mathcal L^{\perp}$. We see that condition (iv) is equivalent to the following linear system of equations having a non-zero solution $t$ for generic $s$:
\begin{equation}\label{bigsystem}
\tr\left(A_i \left(\sum_ks_kB_k\right)
\left(\sum_jt_jA_j\right)
\left(\sum_ls_lB_l\right)\right) = 0,
\quad i=1,\dotsc,c.
\end{equation}
Define the $c\times c\times d\times d$ tensor $T$ by $T_{ijkl} = \tr(A_iB_lA_jB_k)$. Then~\eqref{bigsystem} means
\[
\sum_j\left(\sum_{k,l}s_ks_lT_{ijkl}\right)t_j
= 0, \quad i=1,\dotsc,c.
\]
This is a linear system of equations in the $t_j$ which has a non-zero solution for generic $s$ if and only if $\det( M) = 0$, where $M_{ij}(s)=\sum_{k,l}s_ks_lT_{ijkl}$.
This proves the equivalence of (iv) and (v).
\end{proof}

\begin{cor}
\label{prop:nointerior}
A linear space $\Lcal$ of real symmetric matrices with ML degree 0 has empty intersection with the interior of the positive definite cone. 
\end{cor}

\begin{proof}
If $\Lcal$ contains a positive definite matrix $\K$, then
a generic such matrix satisfies 
$(\K \Lcal^\perp \K) \cap \Lcal \neq \lbrace 0 \rbrace$
by Theorem~\ref{thm:mlDegreeZero}(iv).
After a change of basis under congruence we may assume that $\K$ is the identity.
We get $\Lcal^\perp \cap \Lcal \neq \{0\}$, a contradiction.
\end{proof} 
This shows that any linear space of ML degree 0 only intersects the positive semi-definite cone at the rank deficient matrices. This result is consistent with what is known about the MLE for linear concentration models (\cite[Corollary 2.2]{sturmfels2010multivariate}), namely that the MLE, if it exists, is the unique maximizer of the determinant over the spectrahedron defined by the fiber of the linear sufficient statistics map intersected with $\bbS^n$.
It could be tempting to think that the ML degree 0 linear spaces are exactly those that only intersect the PD cone at the boundary. The following is a counter-example.

\begin{exa}
Consider the one-dimensional linear space $\Lcal$ spanned by
$\left[ \begin{smallmatrix} 1 & 0 \\ 0 & -1 \end{smallmatrix} \right]. $
The reciprocal variety $\Lcal^{-1}$ is equal to $\Lcal$, while $\Lcal^\perp$ consists of matrices
$ \left[\begin{smallmatrix} a & b \\ b & a \end{smallmatrix}\right].$
We see that $\Lcal^{-1} + \Lcal^\perp$ fills the space of $2 \times 2$ symmetric matrices. 
Hence $\Lcal$ has strictly positive ML degree (in fact, ML degree one), but it only intersects the PD cone at zero.  
\end{exa}

We now describe a geometrically interesting subclass of models with ML degree zero.

\begin{rem}\label{rem:geometrically-interesting}
A sufficient condition for a linear space $\Lcal$ to have ML degree zero is if
the reciprocal variety $\Lcal^{-1}$ and the annihilator  $\Lcal^\perp$ lie in a common hyperplane,  by Theorem~\ref{thm:mlDegreeZero}(iii). In other words,
\begin{align}
\label{eq:commonHyperplane}
    \text{there exists } \K \in \Lcal \text{ such that } \mathcal L\subseteq (\{\K\}^{\perp})^{-1}.
\end{align}

Note that $\K$ must be rank deficient: If $\K \in \Lcal$ had full rank, then $\K^{-1} \in \Lcal^{-1}$. But then $\Lcal^{-1} \subseteq \K^\perp$ implies $\K^{-1} \in \K^\perp$, and hence $\tr (\K \K^{-1}) = 0$, a contradiction.
\end{rem}

\begin{lemma}
\label{lem:invisibledeterminant}
For a regular subspace $\mathcal L \subset \bbS^n$, condition~\eqref{eq:commonHyperplane} is equivalent to:
$$
\text{there exists } \K \in \Lcal \text{ such that } \det(P + t\K) = \det(P)
$$
for all $P \in \Lcal$ and all $t\in \CC$.
\end{lemma}
\begin{proof}
Set $B_0 = \K$ and extend to a basis $ \{B_i\}_{i=0}^d$ for $\Lcal$. For all $(t_0,\dotsc,t_d)$ with $\sum_it_iB_i$ invertible, 
condition~\eqref{eq:commonHyperplane} says
\begin{align*}
 0 = \tr\left(B_0 \, \adj(\sum_{i=0}^d t_iB_i)\right) = \frac{d}{dt_0} \det (\sum_it_iB_i).
\end{align*}
Hence the polynomial $\det (\sum_it_iB_i)$ does not depend on $t_0$.
\end{proof}

\begin{exa}
Let $\mathcal L^\perp$ be the singular pencil spanned by 
$\left[\begin{smallmatrix}0&1&0 \\ 1&0&0 \\ 0&0&0\end{smallmatrix}\right]$ and $\left[\begin{smallmatrix}0&0&0 \\ 0&1&0 \\ 0&0&0\end{smallmatrix}\right]$.
Since a generic element of $\mathcal L$ has the form
$\left[\begin{smallmatrix}
t_0&0&t_1 \\ 0&0&t_2 \\ t_1&t_2&t_3
\end{smallmatrix}\right]$,
we see that
\[
\adjHP_{\mathcal L} = \left\lbrace
\left[
\begin{smallmatrix}
-t_2^2 & t_1t_2 & 0 \\ t_1t_2 & t_0t_3 - t_1^2 & -t_0t_2 \\ 0 &-t_0t_2 & 0
\end{smallmatrix}
\right]
\mid
t_0, t_1, t_2, t_3 \in \mathbb C
\right\rbrace.
\]
This shows that $\mathcal L^{-1}$ and $\mathcal L^\perp$ are contained in a common linear space of codimension 2. 
Hence, $\mathcal L$ has ML degree zero by Remark~\ref{rem:geometrically-interesting}.
In fact, up to congruence, this is the only 4-dimensional subspace on $\bbS^3$ with ML degree zero (see Section~\ref{sec:space}.4: the pencil $\Lcal^\perp$ has Segre symbol $[2;;1]$).

Another way to see that $\Lcal$ has ML degree zero is the following. The determinant of a generic element of $\mathcal L$ is $-t_0 t_2^2$. Since it does not depend on $t_1$, we can take $\K = 
\left[\begin{smallmatrix}
0&0&1 \\ 0&0&0 \\ 1&0&0
\end{smallmatrix}\right]$
in Lemma~\ref{lem:invisibledeterminant}. The same reasoning with $t_3$ also applies.
\end{exa}

The same techniques can be used to study the 3-dimensional subspaces $\Lcal \subset \bbS^3$ listed in Section~\ref{sec:space}.3. 

\begin{exa}
Let $\mathcal L  = \{ \left[\begin{smallmatrix}
t_0 & 0 & t_2 \\
0 & t_1 & t_2 \\
t_2 & t_2 & 0
\end{smallmatrix}\right] \mid t_i  \in \mathbb{C} \}$.
It is a 3-dimensional subspace of $\bbS^3$ of type $F$ (see Section~\ref{sec:space}.3).
The determinant is $-(t_0 + t_1)t_2^2$. While the determinant depends on all three variables, the linear forms appearing in it are orthogonal to $t_0 - t_1$. Setting 
$\K =
\left[
\begin{smallmatrix}
1&0&0 \\ 0&-1&0 \\ 0&0&0
\end{smallmatrix}
\right]
$
in Lemma~\ref{lem:invisibledeterminant}
shows that $\mathcal L$ has ML degree zero.
\end{exa}

For \emph{all} regular linear subspaces $\Lcal$ of $\bbS^3$ with ML degree zero (see Section~\ref{sec:space}), it is in fact true that $\Lcal^{-1}$ and $\Lcal^\perp$ are contained in a common hyperplane;
with the following exception (up to congruence).

\begin{exa}
The hyperplane $\Lcal = A^\perp$ with $A = \left[\begin{smallmatrix} 1&0&0 \\ 0&0&0 \\ 0&0&0 \end{smallmatrix}\right]$
has ML degree zero by Proposition~\ref{prop:hyperplane}.
Its reciprocal hypersurface $L^{-1} = \PP \{ \left[\begin{smallmatrix} \sigma_{11}&\sigma_{12}&\sigma_{13} \\ \sigma_{12}&\sigma_{22}&\sigma_{23} \\ \sigma_{13}&\sigma_{23}&\sigma_{33} \end{smallmatrix}\right] \mid \sigma_{22}\sigma_{33} - \sigma_{23}^2 = 0 \}$ is a quadric cone whose vertex set $(\cong \PP^2)$ contains the point $L^\perp$.
Hence, $L^{-1}$ and $L^\perp$ are \emph{not} contained in a common hyperplane, but their join is $L^{-1}$, 
so the ML degree of $\Lcal$ is zero by Theorem~\ref{thm:mlDegreeZero}(iii).

Equivalently, we can compute the $1 \times 1$ matrix $M$ in  Theorem~\ref{thm:mlDegreeZero}(v).
For every matrix $B \in \Lcal$, we have $ABA = 0$, which shows that $M$ is zero.
\end{exa}

In the space of $4 \times 4$ symmetric matrices, 
there are more geometrically interesting regular subspaces with ML degree zero.
We conclude this paper with a class of codimension-two subspaces $\Lcal$, where $L^{-1}$ and $L^\perp$ are not contained in a common hyperplane.

\begin{exa}
Let $L^\perp$ be a tangent line to the variety of rank-one matrices in $\PP \bbS^4$.
After a change of coordinates, we may assume that $$
\Lcal^\perp = \mathrm{span} \left\lbrace 
\left[\begin{smallmatrix} 1&0&0&0 \\ 0&0&0&0 \\ 0&0&0&0 \\ 0 &0 &0 &0 \end{smallmatrix}\right],  \left[\begin{smallmatrix} 0& 1 &0&0 \\ 1& 0&0&0 \\ 0&0&0&0 \\ 0&0&0&0 \end{smallmatrix}\right]
\right\rbrace.
$$ 
The reciprocal variety
is a cubic cone whose vertex set ($\cong \PP^4$) contains  $L^\perp$:
$$L^{-1} =  \{ \Sigma \in \PP \bbS^4 \mid \sigma_{33}\sigma_{44} - \sigma_{34}^2 = 0, \;
\sigma_{23}\sigma_{44} - \sigma_{24}\sigma_{34} = 0, \;
\sigma_{23}\sigma_{34} - \sigma_{33}\sigma_{24} = 0 \}.$$
\end{exa} 

\small
\paragraph{Acknowledgements.}
CA was partially supported by the Deutsche Forschungsgemeinschaft (DFG) in the context of the Emmy
Noether junior research group KR 4512/1-1.
LG was supported by Vetenskapsrådet grant [NT: 2018-03688]. 
KK was supported by the Knut and Alice Wallenberg Foundation within their WASP (Wallenberg AI, Autonomous Systems and Software Program) AI/Math initiative. 
OM was supported by Brummer \& Partners MathDataLab and International Max Planck Research School.

\end{document}